\documentclass{amsart}

\usepackage{tikz}
\usetikzlibrary{matrix}
\usepackage[utf8]{inputenc}
\usepackage{graphicx,subcaption}
\usetikzlibrary{matrix}
\usetikzlibrary{arrows}

\newtheorem{theorem}{Theorem}[section]
\newtheorem{corollary}{Corollary}[theorem]
\newtheorem{lemma}[theorem]{Lemma}

\theoremstyle{definition}

\theoremstyle{remark}

\numberwithin{equation}{section}

\begin{document}

\title[Components of Hilbert scheme in Grassmannians]{Connected Components in the Hilbert Scheme of hypersurfaces in Grassmannians}

\author{See-Hak Seong}
\address{Department of Mathematics,Statistics and CS, University of Illinois at Chicago, Chicago, Illinois 60607}
\email{sseong3@uic.edu}

\begin{abstract}
We show that when $d \geq 3$, the Hilbert scheme $Hilb_{dT+1-\binom{d-1}{2}}(G(k,n))$ has 2 components, even though elements in both components have the same cohomology class. Moreover, we show that the Hilbert scheme associated to the Hilbert polynomial $\binom{T+n}{n}-\binom{T+n-d}{n}$ in Grassmannian has at most 2 connected components.
\end{abstract}

\maketitle

\section*{Contents}
~\\
1. Introduction \hfill 1 \\
2. Preliminaries \hfill 2\\
3. Disconnectivity of $Hilb_{P_{d,m}(T)}(G(k,n))$ for $m>2,d\geq 3$ \hfill 3\\
4. Disconnectivity of $Hilb_{P_{d}(T)}(G(k,n))$ for $d\geq 3$ \hfill 7\\
References

\section{Introduction}
Let $P(T)$ be a polynomial and let $Hilb_{P(T)}(X)$ be the Hilbert scheme associated to the polynomial $P(T)$. In 1966, Hartshorne proved that $Hilb_{P(T)}(\mathbb{P}^n)$ is connected [1]. Moreover, it is well-known that the Hilbert scheme of points of a connected variety is connected [2]. In ``Algebraic Geometry and its Broader Implications'' conference, Ignacio Sols asked whether $Hilb_{P(T)}(G(k,n))$ is connected. We show that $Hilb_{P(T)}(G(k,n))$ is not always connected.

\begin{theorem}
For $d \geq 3$ and $1<k<n-1$, $Hilb_{dT+1-\binom{d-1}{2}}(G(k,n))$ has 2 connected components.
\end{theorem}

We argue by identifying the Hilbert scheme with the space of degree $d$ planar curves (\textit{i.e.} degree $d$ hypersurfaces in $\mathbb{P}^2_{\mathbb{C}}$) in the Grassmannian $G(k,n)$. For simplicity, let
\begin{center}
$P_d(T) := dT+1-\binom{d-1}{2} \in \mathbb{C}[T]$
\end{center}
 We embed the Grassmannian into a projective space $\mathbb{P}^N_{\mathbb{C}}$ via the Pl\"{u}cker embedding and identify $Hilb_{P_d(T)}(\mathbb{P}^N_{\mathbb{C}})$ with the degree $d$ planar curves in $\mathbb{P}^N_{\mathbb{C}}$ by looking at the dimension of tangent spaces at the points of degree $d$ planar curves. Since the dimension of tangent spaces at the points of degree $d$ planar curves in $Hilb_{P_d(T)}(\mathbb{P}^N_{\mathbb{C}})$ agrees with the dimension of the space of degree $d$ planar curves in $\mathbb{P}^N_{\mathbb{C}}$, this tells us that the space of the degree $d$ planar curves in $Hilb_{P_d(T)}(\mathbb{P}^N_{\mathbb{C}})$ is smooth everywhere. By Hartshorne's theorem [1] $Hilb_{P_d(T)}(\mathbb{P}^N_{\mathbb{C}})$ is smooth and connected. Hence, $Hilb_{P_d(T)}(G(k,n))$ consists of degree $d$ planar curves contained in $G(k,n)$. If $d \geq 3$, the plane spanned by the curve lies in $G(k,n)$ as well. The disconnectivity can be proved by using the cohomology classes of the planes in the Grassmannian. 

One of the most interesting properties in Theorem 1.1. is that the curves in both components have the same cohomology class $d \sigma_{n-k,\cdots,n-k,n-k-1}$, which makes it hard to distinguish these two components. However, if we use the same technique and generalize it to the space of degree $d$ hypersurfaces in $\mathbb{P}^m_{\mathbb{C}}$ in the Grassmannian $G(k,n)$, we can distinguish each components by looking at their cohomology classes.

\begin{theorem}
For $d \geq 3, 1<k<n-1$ and $m>2$, let's define  $P_{d,m}(T) = \binom{T+m}{m}-\binom{T+m-d}{m}$. Then the $Hilb_{P_{d,m}(T)}(G(k,n))$ has the following number of connected components :
\begin{center}
$\begin{array}{ll}
2< m \leq k \text{ case} & : \big(\text{number of components in } Hilb_{P_{d,m}(T)}(G(k,n))\big) =2 \\
k < m \leq n-k \text{ case} & : \big(\text{number of components in } Hilb_{P_{d,m}(T)}(G(k,n))\big) =1 \\
n-k< m \text{ case} & : \big(\text{number of components in } Hilb_{P_{d,m}(T)}(G(k,n))\big) =0 \\
\end{array}$
\end{center}
\end{theorem}

In \S 2, we give preliminaries of the Grassmannians and the tangent space of the Hilbert schemes. In \S 3, we prove Theorem 1.2 when $m>2$ and $d \geq 3$. In \S 4, we prove Theorem 1.1 and complete the proof of Theorem 1.2.

\section{Preliminaries}

In this section, we collect basic facts concerning Grassmannians and the tangent space of the Hilbert scheme. Throughout this paper we work over $\mathbb{C}$.

\subsection{The Pl\"{u}cker embedding}
Let $1 \leq k <n$ be integer. Let $G(k,n)$ denote the Grassmannian of $k$-dimensional subspaces in a $n$-dimensional vector space.  We can embed the Grassmannian $G(k,n)$ into projective space $\mathbb{P}^N_{\mathbb{C}}$, where $N= \binom{n}{k}-1$, by the \textit{Pl\"{u}cker embedding}. 

For $1<k<n-1$, the ideal of the image is generated by quadratic polynomials called \textit{Pl\"{u}cker relations} [3].

\subsection{Cohomology classes in Grassmannian}
The cohomology of $G(k,n)$ has an additive basis given by Schubert classes.

Assume that $1<k<n-1$ (otherwise $G(k,n) \cong \mathbb{P}^{n-1}_{\mathbb{C}}$), and take a complete flag $\mathbb{F}_{\bullet}$ in $\mathbb{C}^n$, which is a sequence of subspaces in $\mathbb{C}^n$
\begin{center}
$0 \subset F_1 \subset F_2 \subset \cdots \subset F_n = \mathbb{C}^n$
\end{center}
with $\dim F_i =i$. Also pick a sequence of integers $a=(a_1,\cdots,a_k)$ which satisfies
\begin{center}
$n-k \geq a_1 \geq a_2 \geq \cdots \geq a_k \geq 0$.
\end{center}
Then, the \textit{Schubert variety} $\Sigma_a(F_{\bullet}) $ is defined by
\begin{center}
$\Sigma_a(F_{\bullet}) := \{ V \in G(k,n) : \dim (V \cap F_{n-k+i-a_i}) \geq i \text{ for } 1\leq i \leq k \}$.
\end{center}
We denoted the cohomology class of $\Sigma_a(F_{\bullet})$ by
\begin{center}
$\sigma_a := [\Sigma_a(F_{\bullet})]$,
\end{center}
such a class is called a \textit{Schubert class}, and this class does not depend on $F_{\bullet}$. The Chow ring of Grassmannian $G(k,n)$ is generated by Schubert classes, and the cohomology class of any subvariety of $G(k,n)$ can be expressed as a positive linear combination of Schubert classes.

\subsection{Tangent space of Hilbert scheme}
Let $X$ be a scheme and let $Y$ be a closed subscheme of $X$. Let $P(T)$ be the Hilbert polynomial of $Y$. Then the tangent space of Hilbert scheme $Hilb_{P(T)}(X)$ at the subscheme $Y$ is $H^0(\mathcal{N}_{Y/X})$ (by [4]). In particular, $\mathcal{N}_{Y/X}$ is the normal sheaf to a closed subscheme $Y$
\begin{center}
$\mathcal{N}_{Y/X}=Hom_{\mathcal{O}_Y}(\mathcal{I}/\mathcal{I}^2,\mathcal{O}_Y) = Hom_{\mathcal{O}_X}(\mathcal{I},\mathcal{O}_Y)$
\end{center}
where $\mathcal{I}=\mathcal{I}_{Y/X}$ is the ideal sheaf of $Y$ in $X$. Therefore, if $X= \mathbb{P}^N_{\mathbb{C}}$ and $Y$ is a subvariety defined by a homogeneous ideal $I$, then tangent space of the Hilbert scheme at $Y$ is $Hom_{S}(I,S/I)$, where $S=\mathbb{C}[x_0,\cdots,x_n]$.

\section{Disconnectivity of $Hilb_{P_{d,m}(T)}(G(k,n))$ for $m>2,d\geq 3$}
In this section, we prove Theorem 1.2.

If $X$ is a degree $d$ hypersurface in $\mathbb{P}^m_{\mathbb{C}}$, the Hilbert polynomial of $X$ is $P_{d,m}(T)$, where
\begin{center}
$P_{d,m}(T) = \binom{T+m}{m}-\binom{T+m-d}{m}$.
\end{center}

Hence, $Hilb_{P_{d,m}(T)}(G(k,n))$ contains the space of degree $d$ hypersurfaces of $\mathbb{P}^m_{\mathbb{C}}$ contained in $G(k,n)$. We want to show that every point of $Hilb_{P_{d,m}(T)}(G(k,n))$ is of this form.

\begin{theorem}
There is a one-to-one correspondence between the space of degree $d$ hypersurfaces of $\mathbb{P}^m_{\mathbb{C}}$ in $\mathbb{P}^N_{\mathbb{C}}$ and the Hilbert scheme $Hilb_{P_{d,m}(T)}(\mathbb{P}^N_{\mathbb{C}})$, where $2\leq m\leq N$.
\end{theorem}
\begin{proof}
Let $\mathcal{S}$ be the universal subbundle on $\mathbb{G}(m,N)$. By considering $\mathbb{P}(Sym^d \mathcal{S}^*)$ as a bundle over $\mathbb{G}(m,N)$, each fiber is isomorphic to the space of degree $d$ hypersurfaces in $\mathbb{P}^m_{\mathbb{C}}$. Therefore, by the theorem of the dimension of fibers, we have that
\begin{center}
$\dim_{\mathbb{C}} \mathbb{P}(Sym^d \mathcal{S}^*) = (m+1)(N-m) + \binom{m+d}{m}-1$
\end{center}
and this is irreducible and smooth. Now if we take an arbitrary point in the image of $\mathbb{P}(Sym^d \mathcal{S}^*)$ inside of the Hilbert scheme $Hilb_{P_{d,m}(T)}(\mathbb{P}^N_{\mathbb{C}})$, the point can be expressed as an ideal $I=(f(x_0,\cdots,x_m),x_{m+1},\cdots,x_N)$ in $\mathbb{C}[x_0,\cdots,x_N]$ (up to linear automorphism on $\mathbb{C}[x_0,\cdots,x_N]$). For this point $I$, the tangent space of the Hilbert scheme at $I$ is equal to $Hom(I, S/I)$, where $S=\mathbb{C}[x_0,\cdots,x_N]$. Since $S/I \cong \mathbb{C}[x_0,\cdots,x_m]/(f)$ and this space is the set of graded $S$-module homomorphisms, each morphism can be determined by the image of $f$ and $x_{n+1},\cdots,x_N$. Since $f$ has to map to a degree $d$ homogeneous polynomial in $\mathbb{C}[x_0,\cdots,x_m]/(f)$, this implies $\binom{m+d}{m}-1$ dimensional choices, because the map $f \mapsto \overline{f}$ is the same as the zero map. And for $x_{m+1},\cdots, x_N$, each of $x_i$ should map to a linear homogeneous polynomial  in $\mathbb{C}[x_0,\cdots,x_m]/(f)$, and this gives $(N-m)(m+1)$ dimensional choices. As a result, we compute that
\begin{center}
$\dim_{\mathbb{C}} T_{I}\Big(Hilb_{P_{d,m}(T)}(\mathbb{P}^N_{\mathbb{C}})\Big) = \binom{m+d}{m}-1+(N-m)(m+1)$
\end{center}
Therefore, the dimension of the tangent space of Hilbert scheme agrees with the dimension of $\mathbb{P}(Sym^d \mathcal{S}^*)$. 

In other words, the subspace of degree $d$ hypersurfaces of $\mathbb{P}^m_{\mathbb{C}}$ in the Hilbert scheme is smooth everywhere. If we assume that there is a subscheme $X$ in $\mathbb{P}^N_{\mathbb{C}}$ which is not a degree $d$ hypersurface of $\mathbb{P}^m_{\mathbb{C}}$ but has Hilbert polynomial $P_{d,m}(T)$, there would be a component $M$ in $Hilb_{P_{d,m}(T)}(\mathbb{P}^N_{\mathbb{C}})$ containing $\mathbb{I}(X)$. By, Hartshorne's theorem the Hilbert scheme $Hilb_{P_{d,m}(T)}(\mathbb{P}^N_{\mathbb{C}})$ is connected. Therefore, some component of $Hilb_{P_{d,m}(T)}(\mathbb{P}^N_{\mathbb{C}})$ and the component of the degree $d$ hypersurfaces of $\mathbb{P}^m_{\mathbb{C}}$ must intersect. However, we have proved that the component of degree $d$ hypersurfaces of $\mathbb{P}^m_{\mathbb{C}}$ is smooth everywhere. This is impossible since any intersection of 2 components would be singular. Therefore,  $Hilb_{P_{d,m}(T)}(\mathbb{P}^N_{\mathbb{C}})$ is irreducible and is exactly the same as space of degree $d$ hypersurfaces of $\mathbb{P}^m_{\mathbb{C}}$ contained in $\mathbb{P}^N_{\mathbb{C}}$.
\end{proof}

Recall that $P_d(T) = P_{d,2}(T)$ and hypersurface of $\mathbb{P}^2_{\mathbb{C}}$ is a curve. We get the following corollary.

\begin{corollary}
There is a one-to-one correspondence between the space of degree $d$ planar curves in $\mathbb{P}^N_{\mathbb{C}}$ and the Hilbert scheme $Hilb_{P_{d}(T)}(\mathbb{P}^N_{\mathbb{C}})$, where $2\leq N$.
\end{corollary}

Moreover, for arbitrary subvarieties $X$ in $\mathbb{P}^N_{\mathbb{C}}$, then Theorem 3.1. implies that the space of degree $d$ hypersurfaces of $\mathbb{P}^m_{\mathbb{C}}$ in $X$ agrees with the Hilbert scheme $Hilb_{P_{d,m}(T)}(X)$. Since we can embed the Grassmannian in to a projective space via the Pl\"{u}cker embedding, we get the following results.

\begin{corollary}
There is a one-to-one correspondence between the space of degree $d$ hypersurfaces of $\mathbb{P}^m_{\mathbb{C}}$ in $G(k,n)$ and the Hilbert scheme $Hilb_{P_{d,m}(T)}(G(k,n))$.
\end{corollary}

\begin{corollary}
There is a one-to-one correspondence between the space of degree $d$ planar curves in $G(k,n)$ and the Hilbert scheme $Hilb_{P_{d}(T)}(G(k,n))$.
\end{corollary}

We just identified the Hilbert scheme $Hilb_{P_{d,m}(T)}(G(k,n))$ as the space of degree $d$ hypersurfaces of $\mathbb{P}^m_{\mathbb{C}}$ in $G(k,n)$. With this identification, we will find the number of connected components in $Hilb_{P_{d,m}(T)}(G(k,n))$ and say when it is disconnected.

\begin{theorem}
For $d \geq 3, 1<k<n-1$ and $m>2$, the Hilbert scheme $Hilb_{P_{d,m}(T)}(G(k,n))$ has the following number of connected components :
\begin{center}
$\begin{array}{ll}
2< m \leq k \text{ case} & : \big(\text{number of components in } Hilb_{P_{d,m}(T)}(G(k,n))\big) =2 \\
k < m \leq n-k \text{ case} & : \big(\text{number of components in } Hilb_{P_{d,m}(T)}(G(k,n))\big) =1 \\
n-k< m \text{ case} & : \big(\text{number of components in } Hilb_{P_{d,m}(T)}(G(k,n))\big) =0 \\
\end{array}$
\end{center}
\end{theorem}
\begin{proof}
Without loss of generality, assume that $k \leq n-k$. With the Pl\"{u}cker embedding, we can embed $G(k,n)$ into $\mathbb{P}^N_{\mathbb{C}}$ where $N= \binom{n}{k}-1$. Now let $X$ be a degree $d$ hypersurface of $\mathbb{P}^m_{\mathbb{C}}$ in $G(k,n)$ and let $L$ be an $m$-dimensional projective space which contains $X$ in $\mathbb{P}^N_{\mathbb{C}}$.

Since $1<k<n-1$ and the Pl\"{u}cker relations are quadratic polynomials, $G(k,n)$ is  an intersection of finitely many quadratic hypersurfaces in $\mathbb{P}^N_{\mathbb{C}}$.
\begin{center}
$G(k,n) = \bigcap\limits_{i} Q_i \subset \mathbb{P}^N_{\mathbb{C}}$
\end{center}
From $X \subset G(k,n)$, we can claim that $X \subset Q_i$ for all quadratic hypersurfaces of $\mathbb{P}^N_{\mathbb{C}}$. Now if we think about a linear subspace $L$, then $X$ should be contained in $L \cap Q_i$. If $L \cap Q_i \neq L$, then $L \cap Q_i$ should be a quadratic hypersurface of $\mathbb{P}^m_{\mathbb{C}} \cong L$. However, $\deg(X)=d \geq 3$ and this conflicts with the condition that $X \subset L \cap Q_i$. Therefore, $L \cap Q_i = L$ and this implies that quadratic hypersurface $Q_i$ of $\mathbb{P}^N_{\mathbb{C}}$ must contain the $m$-dimensional projective space $L$ for all $i$. Since $G(k,n)$ is an intersection of all $Q_i$, we can conclude that $L$ is contained in $G(k,n)$.

Now let's look at the cohomology classes of $X$ and $L$. If we denote the cohomology class of $L$ as $[L]$, then Pieri's formula [3] implies that the only possible cohomology classes for $[L]$ are either $\sigma_{n-k,\cdots,n-k,i}~(0\leq i \leq n-k)$ or $\sigma_{n-k,\cdots,n-k,n-k-1,\cdots,n-k-1}$ ($j$ many $n-k$ in the index, where $0 \leq j \leq k$). For each case, the cohomology class of $X$ can be computed as follows :
\begin{center}
$\begin{array}{ll}
[X] = d \cdot \sigma_{n-k,\cdots,n-k,n-k-m+1} & \text{if } [L] = \sigma_{n-k,\cdots,n-k,n-k-m} \\

[X] = d \cdot \sigma_{n-k,\cdots,n-k,n-k-1,\cdots,n-k-1} & \text{if }  [L] = \sigma_{n-k,\cdots,n-k,n-k-1,\cdots,n-k-1} \\

(m-1 \text{ many }n-k-1's) & (m \text{ many }n-k-1's)
\end{array}$
\end{center}
Since $m>2$, we claim that $[X]$ has two possible cases for $2<m \leq k$. If we think about the case when $[X] = d \cdot \sigma_{n-k,\cdots,n-k,n-k-m+1}$, then linear space $L$ should have a cohomology class of $[L] = \sigma_{n-k,\cdots,n-k,n-k-m}$. Recall that if two projective spaces in Grassmannian have the same cohomology classes, then there exists a linear automorphism (Pl\"{u}cker image of a linear automorphism of $\mathbb{C}^n$) on the Grassmannian such that implies isomorphism between those two projective spaces. With this property, the space of $m$-dimensional projective space of cohomology class $\sigma_{n-k,\cdots,n-k,n-k-m}$ in $G(k,n)$ is connected. Moreover, the space of all degree $d$ hypersurfaces in $\mathbb{P}^m_{\mathbb{C}}$ are connected [1]. By combining all those results, we can conclude that the cohomology class $d \cdot \sigma_{n-k,\cdots,n-k,n-k-m+1}$ induces one connected component. In the same way, the other cohomology class induces distinct one connected component. Therefore, for $d\geq 3, 1<k<n-1$ and $2<m \leq k$, $Hilb_{P_{d,m}(T)}(G(k,n))$ has two connected components.

With the same condition on $d$ and $k$, the number of connected component can vary for larger $m$. Because, $[L]$ will not exist for higher $m$ and this will also imply non-existence of $X$. If $k+1 \leq m \leq n-k$, then only possible $[L]$ is
\begin{center}
$[L]= \sigma_{n-k,\cdots,n-k,n-k-1,\cdots,n-k-1}$ ($m$ many $n-k-1$'s)
\end{center}
and this show that $Hilb_{P_{d,m}(T)}(G(k,n))$ has only one connected component. If $m$ is bigger than $n-k$, then such $m$-dimensional linear subspace $L$ will not exist in $G(k,n)$. Therefore, $Hilb_{P_{d,m}(T)}(G(k,n))$ is empty. In summary, for $d \geq 3, 1<k<n-1$
\begin{center}
$\begin{array}{ll}
2< m \leq k \text{ case} & : \big(\text{number of components in } Hilb_{P_{d,m}(T)}(G(k,n))\big) =2 \\
k < m \leq n-k \text{ case} & : \big(\text{number of components in } Hilb_{P_{d,m}(T)}(G(k,n))\big) =1 \\
n-k< m \text{ case} & : \big(\text{number of components in } Hilb_{P_{d,m}(T)}(G(k,n))\big) =0 \\
\end{array}$
\end{center}
As a result, for $d\geq 3, 1<k<n-1$ and $m>2$, $Hilb_{P_{d,m}(T)}(G(k,n))$ has at most 2 connected components
\end{proof}
We did not emphasize this, but it is obvious that two connected components are disconnected, if elements in each component have different cohomology classes. If they are connected, that means two elements in different components induce a flat family and this tells us two elements can deform to each other. However, if this is true, they should have the same cohomology class, but they don't.

\begin{corollary}
For $d\geq 3, 1<k<n-1$ and $2<m \leq k$, $Hilb_{P_{d,m}(T)}(G(k,n))$ is disconnected.
\end{corollary}

For $m=2$ case, there is only one possible cohomology class for $X$ which is $d \cdot \sigma_{n-k,\cdots,n-k,n-k-1}$. However, $Hilb_{P_{d,2}(T)}(G(k,n))=Hilb_{P_{d}(T)}(G(k,n))$ is not connected, because if we consider the cohomology class of $L \cong \mathbb{P}^2_{\mathbb{C}}$ in the Grassmannian, there are two different possible cohomology classes.

Now let's think about the geometry of such a Hilbert scheme. Theorem 3.1 tells us that every subscheme of Hilbert polynomial $P_{d,m}(T)$ is a degree $d$ hypersurface of $\mathbb{P}^m_{\mathbb{C}}$. So, if we could construct the space of linear subspaces $L \cong \mathbb{P}^m_{\mathbb{C}}$ in $G(k,n)$ and then if we construct appropriate vector bundle over such space by using symmetric algebra, then that will give the geometry of Hilbert scheme.

Before talking about the Hilbert scheme, we define the tautological normal bundle $\mathcal{S}$ over a 2 step flag variety $F(k_1,k_2;n)$. If we pick an arbitrary element in the space $F(k_1,k_2;n)$, then it can be expressed as $(W_1,W_2)$, where $W_1 \subset W_2 \subset \mathbb{C}^n$, $\dim W_1=k_1$, and $\dim W_2=k_2$. At this point, we can find $\mathbb{P}(W_2/W_1) \cong \mathbb{P}^m_{\mathbb{C}}$ and by corresponding such projective quotient space as a fiber over the point $(W_1,W_2)$, we can construct the \textit{tautological normal bundle} $\mathcal{S}$ over $F(k_1,k_2;n)$. In other words, if we define the bundle $\mathcal{S}_i$ as a bundle over $F(k_1,k_2;n)$, which each fiber at $(W_1,W_2)$ is $W_i$, for $i=1,2$, then the tautological normal bundle is isomorphic to the quotient bundle $\mathcal{S}_2/ \mathcal{S}_1 \cong \mathcal{S}$. The tautological normal bundle will give a picture of the geometry of the Hilbert scheme.

\begin{theorem}
Assume that $d \geq 3, 1<k \leq n-k$ and $2<m \leq k$. Then, $Hilb_{P_{d,m}(T)}(G(k,n))$ is isomorphic to $\mathbb{P}(Sym^d~\mathcal{S}^*)$, where $\mathcal{S}$ is the tautological normal bundle over the union of flag varieties $F(k-1,k+m;n) \cup F(k-m,k+1;n)$
\end{theorem}
\begin{proof}
In Theorem 3.2, we proved that this Hilbert scheme has two connected components, so let's look at each component step by step.

If we take an arbitrary element $X$ in the Hilbert scheme $Hilb_{P_{d,m}(T)}(G(k,n))$, then there is unique $m$ dimensional linear subspace $L$ in $G(k,n)$ which contains $X$. Moreover, if we fix an $m$-dimensional linear subspace $L$ in $G(k,n)$, we can express the family of degree $d$ hypersurfaces in $L$, by using the $d$-th graded part of the symmetric algebra of the dual of the tautological bundle. Therefore, to observe the geometry of this Hilbert scheme, it is enough to find the space of $m$ dimensional subspaces with the correct tautological bundle on it.

We have shown that there are only two possible cohomology classes for such $L$ in Theorem 3.2. Let's start with the first case.

$\bullet$ (case 1) $[L] = \sigma_{n-k,\cdots,n-k,n-k-m}$

If $[L] = \sigma_{n-k,\cdots,n-k,n-k-m}$, there is a complete flag $F_{\bullet}$ of $\mathbb{C}^n$ such that 
\begin{center}
$L= \Sigma_{n-k,\cdots,n-k,n-k-m}(F_{\bullet})$
\end{center}
where $F_{\bullet}$ is a complete flag with respect to a basis ${v_1,\cdots,v_n}$. If we say that $V\cong \mathbb{C}^k$ is an element in $L \subset G(k,n)$, then the definition of Schubert cycle implies that
\begin{center}
$\dim V \cap F_{i} \geq i$ for $1\leq i \leq k-1$

$\dim V \cap F_{k+m} \geq k$
\end{center}
Therefore, with the Pl\"{u}cker embedding, $V$ can be expressed as following :
\begin{center}
$V \mapsto v_1 \wedge \cdots \wedge v_{k-1} \wedge (a_kv_k + \cdots + a_{k+m}v_{k+m})$

where $a_k,\cdots,a_{k+m} \in \mathbb{C}$ and $(a_k,\cdots,a_{k+m}) \neq (0,\cdots,0)$

(by corresponding $V$ with $[a_k: \cdots : a_{k+m}] \in \mathbb{P}^m_{\mathbb{C}}$, we get $L \cong \mathbb{P}^m_{\mathbb{C}}$)
\end{center}
So, when $[L] = \sigma_{n-k,\cdots,n-k,n-k-m}$, then we can express all the points in $L$ as $v_1 \wedge \cdots \wedge v_{k-1} \wedge (a_kv_k + \cdots + a_{k+m}v_{k+m})$. Moreover, we can determine such $L$ by picking a $k-1$ dimensional space $Span \{v_1,\cdots,v_{k-1} \}$ and a $k+m$ dimensional space $Span \{v_1,\cdots,v_{k+m} \}$. Therefore, the space of such $L$ is the flag variety $F(k-1,k+m;n)$. 

With the tautological bundle $\mathcal{S}$ above, if we pick an arbitrary point $(W,W') \in F(k-1,k+m;n)$ then the bundle $\mathbb{P}(Sym^d~\mathcal{S}^*)$ defines degree $d$ homogeneous polynomial defined over $\mathbb{P}(W'/W)\cong \mathbb{P}^m_{\mathbb{C}}$ which corresponds to the degree $d$ hypersurface in $\mathbb{P}^m_{\mathbb{C}}$. Therefore, the connected component in $Hilb_{P_{d,m}(T)}(G(k,n))$ which corresponds to the cohomology class $d \cdot \sigma_{n-k,\cdots,n-k,n-k-m+1}$ is isomorphic to the bundle $\mathbb{P}(Sym^d~\mathcal{S}^*) \to F(k-1,k+m;n)$.

$\bullet$ (case 2) $[L] = \sigma_{n-k,\cdots,n-k,n-k-1,\cdots,n-k-1}$ $(m \text{ many }n-k-1's)$

For this case, with the same method as above, we can express $L$ as following
\begin{center}
$L= \Sigma_{n-k,\cdots,n-k,n-k-1,\cdots,n-k-1}(F_{\bullet})$.
\end{center}
By using the definition of Schubert varieties, if we take an arbitrary element $V \in L$, then
\begin{center}
$\dim V \cap F_{i} \geq i$ for $1 \leq i \leq k-m$

$\dim V \cap F_{i+1} \geq i$ for $k-m+1 \leq i \leq k$
\end{center}
Therefore, each $V \cong \mathbb{C}^k$ can be expressed as
\begin{center}
$e_1 \wedge \cdots \wedge e_{k-m} \wedge (*e_{k-m+1} + *e_{k-m+2}) \wedge \cdots \wedge (*e_{k-m+1} + \cdots + *e_{k+1})$

where $*$ are constants in $\mathbb{C}$
\end{center}
Therefore, the space of $L$ can be defined as $F(k-m,k+1;n)$ in this case. With the same process above in (case 1), we can deduce the tautological normal bundle $\mathcal{S}$ and the bundle $\mathbb{P}(Sym^d~\mathcal{S}^*) \to F(k-m,k+1;n)$ is isomorphic to the connected component in $Hilb_{P_{d,m}(T)}(G(k,n))$ which corresponds to the cohomology class $d \cdot \sigma_{n-k,\cdots,n-k,n-k-1,\cdots,n-k-1}$.

By investigating those two cases, we conclude that $Hilb_{P_{d,m}(T)}(G(k,n))$ is isomorphic to $\mathbb{P}(Sym^d~\mathcal{S}^*)$ over the flag variety $F(k-1,k+m;n) \cup F(k-m,k+1;n)$.
\end{proof}

\begin{corollary}
Assume that $d \geq 3, 1<k<n-k$ and $k<m \leq n-k$. Then, $Hilb_{P_{d,m}(T)}(G(k,n))$ is isomorphic to $\mathbb{P}(Sym^d~\mathcal{S}^*)$, where $\mathcal{S}$ is the tautological normal bundle over the flag variety $F(k-m,k+1;n)$
\end{corollary}

\begin{proof}
Use Theorems 3.2 and 3.3.
\end{proof}

\section{Disconnectivity of $Hilb_{P_{d}(T)}(G(k,n))$ for $d\geq 3$}

In Theorem 3.2, we proved that $Hilb_{P_{d,m}(T)}(G(k,n))$ has at most 2 connected components and each component has distinct cohomology classes. In the case of $m=2$, when $P_{d,m}(T)=P_{d}(T)$, Corollary 3.1.3 implies that $Hilb_{P_{d}(T)}(G(k,n))$ is the space of degree $d$ planar curves, and all degree $d$ planar curves has a unique cohomology class $d \cdot \sigma_{n-k,\cdots,n-k,n-k-1}$. However, it is not true that the Hilbert scheme $Hilb_{P_{d}(T)}(G(k,n))$ is connected. In this section, we will show when $d\geq 3$ this Hilbert scheme is disconnected.

Most of the proof is similar to the proof of Theorem 3.2. Instead of using $X$ and $L$, let $C$ be a degree $d$ plane curve in $G(k,n)$ and $P \cong \mathbb{P}^2_{\mathbb{C}}$ the unique plane in $G(k,n)$ such that contains $C$ (existence and uniqueness of $P$ is proved in Theorem 3.2).

\begin{lemma}
Let $C$ and $C'$ be degree $d$ planar curves in $G(k,n)$ such that $I_C,I_{C'} \in Hilb_{P_{d}(T)}(G(k,n))$ are in the same connected components of the Hilbert scheme. If we say $P,P'$ are planes in $G(k,n)$ which satisfy $C \subset P$ and $C' \subset P'$, then 
\begin{center}
(cohomology class of $P$ in $G(k,n)$) = (cohomology class of $P'$ in $G(k,n)$)
\end{center}
\end{lemma}
\begin{proof}
Let $I_P, I_{P'}$ be ideals on $\mathbb{P}^N_{\mathbb{C}}$ corresponding to $P$ and $P'$ respectively (where $G(k,n)$ embeds to $\mathbb{P}^N_{\mathbb{C}}$ by Pl\"{u}cker embedding). Since $C \subset P$ and $C' \subset P'$, we can find inclusions $I_P \hookrightarrow I_C$ and $I_{P'} \hookrightarrow I_{C'}$. Also $I_C$ and $I_{C'}$ are connected, there is a flat family $\{I_t\}$ of ideals on $H_{P_{d}(T)}(G(k,n))$ which connects $I_C$ and $I_{C'}$ (either $I_0=I_C, I_1=I_{C'}$ or $I_0=I_{C'}, I_1=I_C$).
\begin{center}
\begin{tikzpicture}[every node/.style={midway}]
        \matrix[column sep={12em,between origins}, row sep={5em}] at (0,0) {
\node(11) {$I_C$};&\node(12) {$I_{C'}$};\\
\node(21) {$I_P$};&\node(22) {$I_{P'}$};\\
        };
        \draw[right hook->] (21) -- (11) ;
        \draw[right hook->] (22) -- (12) ;
        \draw[->] (11) -- (12) node[above]{\textit{connected by a flat family}};
\end{tikzpicture}
\begin{tikzpicture}[every node/.style={midway}]
        \matrix[column sep={12em,between origins}, row sep={5em}] at (0,0) {
\node(11) {$C$};&\node(12) {$C'$};\\
\node(21) {$P$};&\node(22) {$P'$};\\
        };
        \draw[right hook->] (11) -- (21) ;
        \draw[right hook->] (12) -- (22) ;
        \draw[->] (11) -- (12) node[above]{\textit{continuous deformation}};
\end{tikzpicture}
\end{center}
By the way, a flat family $\{I_t\}$ between $I_C$ and $I_{C'}$ induces a deformation from a plane cubic $C$ to $C'$. Moreover, $P$ and $P'$ are unique planes in $G(k,n)$ such that $C \subset P$ and $C' \subset P'$. Therefore, a flat family between $I_C$ and $I_{C'}$ induces a continuous deformations form $P$ to $P'$ in $G(k,n)$. From basic algebraic topology, we can conclude that the cohomology class of $P$ and $P'$ are the same.
\end{proof}

\begin{theorem}
For $d \geq 3$ and $1<k<n-1$, $Hilb_{P_{d}(T)}(G(k,n))$ has two connected components.
\end{theorem}
\begin{proof}
Most part of the proof is similar to the proof of Theorem 3.2. But the difference is that each connected component is determined by the cohomology class of $P$ instead of the cohomology class of $C$. Since there are two possible cohomology classes $\sigma_{n-k,\cdots,n-k,n-k-2}$ and $\sigma_{n-k,\cdots,n-k,n-k-1,n-k-1}$ for $P$, by using the lemma above, we can conclude that $Hilb_{P_{d}(T)}(G(k,n))$ has two connected components.
\end{proof}

The most interesting result of Theorem 4.2 is that one cohomology class in the Hilbert scheme could be separated into different connected components, although such phenomenon did not appear in Theorem 3.2.

For the geometry of $Hilb_{P_d(T)}(G(k,n))$, we can use the same method that we used in Theorem 3.3. If we combine with the theorem above then we can also understand the geometry of $Hilb_{P_d(T)}(G(k,n))$.

\begin{corollary}
Assume that $d \geq 3, 1<k<n-k$. Then, $Hilb_{P_{d}(T)}(G(k,n))$ is isomorphic to $\mathbb{P}(Sym^d~\mathcal{S}^*)$, where $\mathcal{S}$ is the tautological normal bundle over the union of flag varieties $F(k-1,k+2;n) \cup F(k-2,k+1;n)$
\end{corollary}

\bibliographystyle{amsplain}

\end{document}